\documentclass{amsart}
\usepackage{lmodern}

\usepackage[latin9]{inputenc}
\usepackage{amsthm}
\usepackage{amssymb}
\usepackage[unicode=true,pdfusetitle,
 bookmarks=true,bookmarksnumbered=false,bookmarksopen=false,
 breaklinks=false,pdfborder={0 0 1},backref=section,colorlinks=false]
 {hyperref}
\hypersetup{
 bookmarks,colorlinks,pagebackref}

\makeatletter
\numberwithin{equation}{section}
\numberwithin{figure}{section}

\usepackage{bbm}\usepackage{tikz}

\usepackage{times}\usepackage{amsfonts}\usepackage{amscd}\usepackage{amsxtra}\usepackage{latexsym}\usepackage{multirow}\usepackage{amsthm}\usepackage{mathtools}\usepackage{cjhebrew}\input{xy}
\xyoption{all}

    \newcommand{\sbt}{\,\begin{picture}(-1,1)(-1,-3)\circle*{2}\end{picture}\ }
\newcommand{\sarc}{\mathrel{\ooalign{$\nabla$\cr
  \hidewidth\raise.3ex\hbox{$\sbt\mkern5mu$}\cr}}}

\newcommand{\nn}{\,\begin{picture}(-1,1)(-1,-3)\scalebox{.5}{n}\end{picture}\ }
\newcommand{\narc}{\mathrel{\ooalign{$\nabla$\cr
  \hidewidth\raise.05ex\hbox{$\nn\mkern7mu$}\cr}}}

\theoremstyle{plain}

\newtheorem*{conjectuur*}{Conjecture}

\swapnumbers
\newtheorem{theorem}[subsection]{Theorem}

\newtheorem{lemma}[subsection]{Lemma}
\newtheorem{proposition}[subsection]{Proposition}

\theoremstyle{definition}
\newtheorem{definition}[subsection]{Definition}

\newtheorem{example}[subsection]{Example}

\theoremstyle{remark}

\newtheorem{remark}[subsection]{Remark}

\swapnumbers

\newcommand{\emptyprop}{q}

\newcommand{\ulim}[1]{\mbox{ulim}{#1}}
\renewcommand{\dim}[1]{\mbox{dim}{#1}}
\newcommand{\sX}[1]{\mathcal{X}_{#1}}
\newcommand{\sY}[1]{\mathcal{Y}_{#1}}
\newcommand{\sZ}[1]{\mathcal{Z}_{#1}}
\newcommand{\bA}{\mathbb{A}}

\newcommand{\bX}{\mathbb{X}}

\newcommand{\sset}{\categ{sSet}}

\newcommand{\sch}[1]{\categ{Sch}_{#1}}
\newcommand{\set}{\categ{Set}}

\newcommand{\ssieves}[1]{\categ{sSieve}_{#1}}
\newcommand{\nsieves}[2]{\categ{s}^{#2}\categ{Sieve}_{#1}}
\newcommand{\mot}{\text{\fontsize{16}{26}\selectfont\cjRL{M}}}

\newcommand{\mmot}{\categ{Mes}\text{\fontsize{16}{26}\selectfont\cjRL{M}}}

\newcommand{\sF}[1]{\mathcal{F}({#1})}
\newcommand{\spec}[1]{\operatorname{Spec}(#1)}

\newcommand{\sS}[3]{\mathcal{S}_{#1}^{#2}{(#3)}}
\newcommand{\sP}[3]{\mathcal{P}_{#1}^{#2}{(#3)}}
\newcommand{\sK}[3]{\mathcal{K}_{#1}^{#2}{(#3)}}
\newcommand{\sC}[3]{\mathcal{C}_{#1}^{#2}{(#3)}}
\newcommand{\sT}[3]{\mathcal{T}_{#1}^{#2}{(#3)}}
\newcommand{\bsT}[3]{\bar{\mathcal{T}}_{#1}^{#2}{(#3)}}
\newcommand{\bsC}[3]{\bar{\mathcal{C}}_{#1}^{#2}{(#3)}}
\newcommand{ \fn}{\mathfrak{n}}

\newcommand{\complet}[1]{\widehat {#1}}

\newcommand{\into}{\hookrightarrow}

\newcommand{\maxim}{\mathfrak m}
\newcommand{\nat}{\mathbb N}

\newcommand{\zet}{\mathbb Z}

\newcommand{\commdiagram}[9][]{%
\begin{equation}
{\newcommand{\tmpprop}{#1q} 
\if\tmpprop\emptyprop \relax\else \label{#1}\fi}
\begin{aligned}%
\mbox{
\begin{picture}(130,90)%
\put(120,70){\vector( 0,-1){50}}%
\put(10,80){\vector( 1, 0){100}}%
\put(0,70){\vector( 0,-1){50}}%
\put(10,10){\vector( 1, 0){100}}%
\put(115,80){\makebox(0,0)[l]{$#4$}}%
\put(5,80){\makebox(0,0)[r]{$#2$}}%
\put(115,10){\makebox(0,0)[l]{$#9$}}%
\put(5,10){\makebox(0,0)[r]{$#7$}}%
\put(-3,50){\makebox(0,0)[r]{$#5$}}
\put(123,50){\makebox(0,0)[l]{$#6$}}
\put(60,3){\makebox(0,0)[c]{$#8$}}
\put(60,88){\makebox(0,0)[c]{$#3$}}
\end{picture}}
\end{aligned}
\end{equation}}

\newcommand{\commtrianglefront}[7][]{%
\begin{equation}
{\newcommand{\tmpprop}{#1q} 
\if\tmpprop\emptyprop \relax\else \label{#1}\fi}
\begin{aligned}%
\mbox{
\begin{picture}(120,80)%
\put(55,68){\vector(-1,-2){30}}
\put(65,68){\vector(1,-2){30}}
\put(30,5){\vector(1,0){60}}
\put(60,75){\makebox(0,0)[c]{$#2$}}
\put(25,5){\makebox(0,0)[r]{$#4$}}
\put(95,5){\makebox(0,0)[l]{$#6$}}
\put(60,0){\makebox(0,0)[c]{$#5$}}
\put(37,43){\makebox(0,0)[r]{$#3$}}
\put(83,43){\makebox(0,0)[l]{$#7$}}
\end{picture}}
\end{aligned}
\end{equation}}

\newcommand{\commtriangleback}[7][]{%
\begin{equation}
{\newcommand{\tmpprop}{#1q}
\if\tmpprop\emptyprop \relax\else \label{#1}\fi}
\begin{aligned}%
\mbox{
\begin{picture}(120,80)%
\put(55,70){\vector(-1,-2){30}}
\put(65,70){\vector(1,-2){30}}
\put(30,5){\vector(1,0){60}}
\put(60,75){\makebox(0,0)[c]{$#2$}}
\put(25,5){\makebox(0,0)[r]{$#6$}}
\put(95,5){\makebox(0,0)[l]{$#4$}}
\put(60,0){\makebox(0,0)[c]{$#5$}}
\put(37,43){\makebox(0,0)[r]{$#7$}}
\put(83,43){\makebox(0,0)[l]{$#3$}}
\end{picture}}
\end{aligned}
\end{equation}}

\usepackage{mathabx}

\newcommand{\fat}{\mathfrak z}

\newcommand{\categ}[1]{\mathbbmss{#1}}

\newcommand{\limfat}[1]{\complet{\categ{Fat}}_{#1}}

\newcommand{\sieves}[1]{\categ{Sieve}_{#1}}

\newcommand{\fld}{\kappa}

\newcommand{\class}[1]{{[ #1]}}

\newcommand{\grot}[1]{{\mathbf {Gr}(#1)}}

\newcommand{\fatpoints}[1]{\categ {Fat}_{#1}}

\newcommand{\lef}{{\mathbb L}}

\makeatother

\begin{document}

\title{Integrable functions within the theory of schemic motivic integration}

\author{Andrew R. Stout}

\address{Andrew R. Stout\\
 Graduate Center, City University of New York\\
 365 Fifth Avenue\\
10016\\
U.S.A. \& Universit� Pierre-et-Marie-Curie \\
 4 place Jussieu \\
75005\\
Paris, France}

\email{astout@gc.cuny.edu}

\maketitle

\begin{abstract}
We develop integrable functions within the theory of relative simplicial
motivic measures. 
\end{abstract}
\setcounter{tocdepth}{1} \tableofcontents{}

\section{Introduction}

In \cite{me2}, the author developed the theory of measurable limit simplicial motivic sites (measurable motivic sites for short) in order to shed light on computational issues encountered in \cite{me1}. The simpliical approach allowed for the prospect of defining motivic integration over derived stacks, yet at another level it allowed the author to investigate classes of functions not yet appearing in the literature: permissible functions, total functions, and cach\'e functions along with an integrability condition. 
This work was heavily inspired by the definable approach to motivic integration as laid bare in the seminal work \cite{CL} of R. Cluckers and F. Loeser. 

Schemic motivic integration was created by H. Schoutens in  \cite{schmot1} and \cite{schmot2}. One should consult these works for terminology and background on this field. I  gave a quick introduction to the subject in \cite{me2} within the framework of simplicial sieves. This current work heavily relies upon the results of these three papers.

The motivation behind this work is the hope of producing a general change of variables formula in the context of schemic motivic integration. This work was partially supported by the chateaubriand fellowship, Prof. F. Loeser, and DSC research grant.

\section{Permissible schemic functions} \label{2}

Let $X \in \sch \fld$. By a \textit{function} $f : X \to \nat$, we will mean a collection of set maps 
\begin{equation*}
f(\maxim)  : X^{\circ}(\maxim) \to \nat \ \mbox{for each} \ \maxim \in \fatpoints \fld
\end{equation*}
such that for any embedding $i : \maxim' \into \maxim$, then $f(\maxim')(i^{*}x)=f(\maxim)(x)$. Note that this asures us that 
the image of $f(\maxim')$ is naturally included into the image of $f(\maxim)$.
 Given one of these so-called functions, we can associate a contravariant functor $\Gamma_f : \fatpoints\fld \to \set $ defined 
by 
\begin{equation}
\Gamma_f(\maxim):= \Gamma_{f(\maxim)}:= \{(x,n) \in X^{\circ}(\maxim)\times \nat \mid  f(\maxim)(x) = n\} \ .
\end{equation}
Note tha a morphism $j: \maxim' \to \maxim$ induces a map of sets $j^*:X^{\circ}(\maxim)\to X^{\circ}(\maxim')$ by sending a 
point $x: \maxim \to X$ to the point $j^*x:\maxim' \to X$ defined by the composition $x \circ j$. In the same way, 
this induces a map of sets $j^* : \Gamma_f(\maxim) \to \Gamma_f(\maxim')$ by sending $(x,f(\maxim)(x))$ to 
$(j^*x,f(\maxim')(j^*x))$.
For each $\maxim \in \fatpoints\fld$, this also gives rise to an assignment from $\Delta^{\circ}$ to $\set$ by sending a simplicial complex $[n]$ to the set 
\begin{equation}
\Gamma_f(\maxim)([n]):=f(\maxim)^{-1}(n) =\Gamma_f(\maxim) \cap( X^{\circ}(\maxim)\times\{n\}) \ .
\end{equation}
From this, we can see that $\Gamma_f(\maxim)([n])$ can be identified with a subset of $X^{\circ}(\maxim)$. The inclusion of $\Gamma_f(\maxim)([n])$ into $X^{\circ}(\maxim) $ will be natural in fat points. Thus,  for each $[n] \in \Delta^{\circ}$, the presieve $\Gamma_f(-)([n])$ is actually a sieve with $X$ as its ambient space.

\begin{definition} 
Let $X \in \sch \fld$ and let $f : X \to \nat$ be a function. We will say that $f$ is a {\it  permissible function} if the assignment 
from
$\Delta^{\circ}$ to $\sieves\fld$ defined by sending $[n]$ to $\Gamma_f(-)([n])$ is in fact a functor. In this case,  
$\Gamma_f $
will denote the corresponding element in $\ssieves\fld$.
\end{definition}

\begin{example}
For any element $X$ of $\sch\fld$, we have the so-called {\it positive constant functions} which are precisely the functions such 
that for all $\maxim \in \fatpoints\fld$, $f(\maxim)$ is equal to some fixed $n$ on $X^{\circ}(\maxim)$. In this case, we have 
the following simple formula 
\begin{equation}
\tau_n(\Gamma_f) = X \ \mbox{and where} \ \tau_m(\Gamma_f) = \emptyset \ \mbox{when} \ m \neq n \ .
\end{equation}
Thus, it is immediate that every positive constant function is a  permissible function. We will denote the collection of positive constant functions on 
$X$
by $\sK{}{+}{X}$.
\end{example}

Given a scheme $X \in \sch\fld$, we form the set of all permissible functions and denote it by $\sS{}{+}{X}$. By definition, $\sS{}{+}{X}$ can be identified with a subset of $\ssieves\fld$ via $f \mapsto \Gamma_f$. As such, the operations of $\times, \ \sqcup, \ \cup$ and $\cap$ restrict to operations on $\sS{}{+}{X}$. Here, these operations are taking place in $\sset$. For example, 
\begin{equation}
(\Gamma_f \cup \Gamma_g)(\maxim)([n]) = \Gamma_f(\maxim)([n])\cup \Gamma_g(\maxim)([n]) \ .
\end{equation}

We can introduce even more operations on the set $\sS{}{+}{X}$. Given two elements $f$ and $g$ in $\sS{}{+}{X}$, we form the function
$f+g$ defined by 
\begin{equation} \label{addsch}
(f+g)(\maxim)(x):=(f(\maxim)+g(\maxim))(x) = f(\maxim)(x)+g(\maxim)(x)
\end{equation}
for each $x\in X^{\circ}(\maxim)$ and for each $\maxim\in\fatpoints\fld$. Likewise, we form the function
$f\cdot g$ defined by 
\begin{equation}\label{multsch}
(f\cdot g)(\maxim)(x):=(f(\maxim)\cdot g(\maxim))(x) = f(\maxim)(x)\cdot g(\maxim)(x)
\end{equation}
for each $x\in X^{\circ}(\maxim)$ and for each $\maxim\in\fatpoints\fld$.

\begin{proposition}\label{persemi}
Let $X\in \sch\fld$. The set $\sS{}{+}{X}$ is a semiring with respect to the operations $+$ and $-$ defined in formulas 2.4 and 2.5, respectively.
\end{proposition}

\begin{proof}
The only thing that remains unclear is if $+$ and $-$ are closed operations -- meaning, if $f , g \in \sS{}{+}{X}$, then $f+g, f\cdot g \in \sS{}{+}{X}$. Therefore, we show that $f+g$ is a permissible function (as the proof for $f\cdot g$ will follow along the same lines). To wit,
we claim that the assignment  $[n]\mapsto\Gamma_{f+g}(-)([n])$ is a functor. Note that, for each $\maxim \in \fatpoints\fld$,
we have the following decomposition 
\begin{equation}
\Gamma_{f+g}(\maxim)([n]) = \bigcup_{t=0}^{n}(\Gamma_{f}(\maxim)([t]) \cup \Gamma_{g}(\maxim)([n-t])) \ .
\end{equation}
Since $f$ and $g$ are permissable, we have expressed the assignment as a union of simplicial sieves which makes it indeed a functor.
\end{proof}

Because of Proposition \ref{persemi}, we may associate to each element $X$ of $\sch\fld$ its {\it ring of permissible functions}  denoted by $\sS{}{}{X}$, which is formed by taking the associated grothendieck ring of $\sS{}{+}{X}$ with respect to the $+$ operation.

\begin{proposition}\label{persheaf}
Let $X \in \sch\fld$. The assignment which sends any open set $U$ of $X$ to the ring $\sS{}{}{U}$ (resp., the semiring $\sS{}{+}{U}$) is a sheaf of rings (resp., a sheaf of semirings) on the scheme $X$.
\end{proposition}

\begin{proof}
For two open sets $U \subset V$ of a scheme $X$, we have the restriction  morphism $\sS{}{+}{V} \to \sS{}{+}{U}$ given by sending $f$ to the unique element of $\sS{}{+}{U}$ whose graph is of the form $\Gamma_f \cap (U)_{\bullet}$. Note that this is also enough to define the restriction homomorphim from  $\sS{}{}{V} $ to  $\sS{}{}{U}$.  The fact that $\sS{}{+}{-}$ and $\sS{}{}{-}$  are presheaves on $X$ is immediate. 

Let $U_i$ be an arbitrary cover of some open set $U$ of $X$. We wish to show that the first arrow in the  
diagram of semiring homomorphisms
\begin{equation}
\sS{}{+}{U} \to \prod_i \sS{}{+}{U_i} \genfrac{}{}{0pt}{0}{\to}{\to} \prod_{i, j}\sS{}{+}{U_i\cap U_j} 
\end{equation}
is an equalizer. This is straightforward. 
Letting $f_i \in  \sS{}{+}{U_i}$ be such that $f_i|_{U_j} = f_j|_{U_i}$ in $\sS{}{+}{U_i\cap U_j}$, we can just form the function $f$ on $U$ defined by sending an $x \in U^{\circ}(\maxim)$ to $f_i(\maxim)(x)$ for any $i$ such that $x\in (U_i)^{\circ}(\maxim)$ for all $\maxim \in \fatpoints\fld$. Uniqueness of this gluing holds because the corresponding graphs will be the same.

 We will let $\mbox{gr}(U) : \sS{}{+}{U} \to \sS{}{}{U}$ denote the grothendieck semiring homomorphism with respect to the operation $+$. Let $f_i \in  \sS{}{}{U_i}$ be such that $f_i|_{U_j} = f_j|_{U_i}$. As  $\mbox{gr}(U_i)$ is surjective for each $i$, these functions lift to function $g_i \in  \sP{}{+}{U_i}$ such that  
\begin{equation}
\mbox{gr}(U_i\cap U_j)(g_i|_{U_j})=\mbox{gr}(U_i)(g_i)|_{U_j} 
= \mbox{gr}(U_j)(g_j)|_{U_i} =\mbox{gr}(U_i\cap U_j)(g_j|_{U_i}) \ .
\end{equation}
Thus, there are a functions $m_{ij}, n_{ij} \in \sS{}{+}{U_i\cap U_j}$ for each $i$ and $j$ such that $g_i|_{U_j} + m_{ij} = g_j|_{U_i}+ n_{ij}$. Thus, by gluing as before we arrive at unique functions $g, \  m, \  n \in \sS{}{+}{U}$ such that 
$g + m = g + n$. From this we find that $\mbox{gr}(U)(m) =\mbox{gr}(U)(n)$ which means that the function $f = \mbox{gr}(U)(g)$ is such that $f|_{U_i} = f_i$. Uniqueness of this gluing follows from a similar argument.
\end{proof}

\begin{remark}
The assignment which sends an open set $U$ of $X$ to the grothendieck semiring homomorphism $\mbox{gr}(U) : \sS{}{+}{U} \to \sS{}{}{U}$ is a morphism of semiring sheaves. Therefore, the stalk $\mathcal{S}_x$ of  $\sP{}{}{-}$ at $x \in X$ is the grothendieck ring with respect to $+$ of the stalk $\mathcal{S}_x^+$ of  $\sP{}{+}{-}$ at $x \in X$.
\end{remark}

One way to describe the functions in $\sS{}{}{X}$ is to consider pairs $(f,g)$ where $f$ and $g$ are elements of  $\sS{}{+}{X}$  and where we mod out by the following equivalent relation:
$(f_1,g_1)\sim(f_2,g_2)$ if there exist some function $h \in\sS{}{+}{X}$ such that $f_1 +g_2 +h = f_2 + g_1 + h$. Here one should think of $f$ as the positive part and $g$ as the negative part.
 Further, this just means that elements of  $\sS{}{}{X}$ are \textit{functions} $F$ from $X$ to $\nat^2$ modulo $\sim$. By this, we mean that for each $\maxim \in \fatpoints\fld$ we have a function from $X^{\circ}(\maxim)$ to $\nat^2$ which is functorial with respect to inclusions of fat points as before -- with the added property that the assignment from $\Delta^{\circ}\times \Delta^{\circ}$ to $\sieves\fld$ given by
\begin{equation}
([n],[m])\mapsto\Gamma_F(([n],[m])) 
\end{equation}
 where $\Gamma_F(([n],[m]))$ is the $\fld$-sieve defined by sending each $\maxim \in \fatpoints\fld$ to the set
\begin{equation}
\Gamma_{F(\maxim)}\cap (X^{\circ}(\maxim)\times\{(n,m)\}) \ 
\end{equation}
is a functor -- i.e., $\Gamma_F(-,-)$ is a bisimplicial $\fld$-sieve. Since we therefore have use for such objects, we will denote by $\nsieves{\fld}{n}$  the category of all covariant functors $\prod_{j=1}^n \Delta^{\circ} \to  \sieves\fld$, and we call such objects $n$-{\it simplicial} $\fld$-{\it sieves} (or, bisimplicial $\fld$-sieves and trisimplicial $\fld$-sieves for $n=2$ and $n=3$, respectively.) The category $\nsieves{\fld}{n}$ is equivalent to the category of covariant functors $\Delta^{\circ}\to\nsieves{\fld}{n-1}$ for any $n > 0$.

\begin{example}
It is easy to check that $\sK{}{+}{X}$ is a sub-semiring of $\sS{}{+}{X}$ for $X\in\sch\fld$. We define {\it the ring of constant functions on} $X$, denoted by $\sK{}{}{X}$, to be the image of $\sK{}{+}{X}$ in $\sS{}{}{X}$ under the 
grothendieck semiring homomorphism $\mbox{gr}(X)$. We can identify elements  $\sK{}{}{X}$ with a collection of functions
$f(\maxim) : X^{\circ}(\maxim) \to \zet$ defined by $f(\maxim)(x)=n$ for all $x\in X^{\circ}(\maxim)$, for all $\maxim \in \fatpoints\fld$, and for some fixed $n \in\zet$ . This is because the graphs of the positive and the negative part are instantly sieves.
\end{example}

\begin{remark}
In fact, the assignment which sends an open set $U$ of $X$ to $\sK{}{+}{U}$ (resp., $\sK{}{}{U}$) is a sheaf of semirings (resp. a sheaf of rings) on $X$. Moreover, $\sS{}{+}{-}$ (resp., $\sS{}{}{-}$) is a sheaf of $\sK{}{+}{-}$-semialgebras (a sheaf of $\sK{}{}{-}$-algebras).
\end{remark}

\section{Permissible simplicial functions} \label{3}

Let $\sX\ $ of $\nsieves{\fld}{n}$ . By a {\it point} of $\sX{}(\maxim)$, we mean a tuple $(x,\bar m)$ such that $x\in  \sX{[\bar m]}(\maxim)$ and $\bar m \in \nat^n$. By a {\it function} $f:\sX{}\to\nat$, we mean a collection $f(\maxim)$ of set-theorectic functions from the set of points of $\sX{\maxim}$ to $\nat$ with the property that $f(\maxim')(i^{*}x,\bar m)=f(\maxim)(x,\bar m)$ for all $\bar m \in \nat^n$ and all $x \in \sX{[\bar m]}(\maxim)$
whenver we have an embedding $i : \maxim' \to \maxim$ in $\fatpoints\fld$. Often we will suppress the brackets around $\bar m$.

\begin{example}
 Consider the constant functor $(-)_{\bullet} : \sieves\fld \to\nsieves{\fld}{n}$ which sends a sieve $\sX{}$ to the $n$-simplicial sieve $(\sX{})_{\bullet}$ which is defined as the functor which sends $\sigma\in\prod_{j=1}^{n}\Delta^{\circ}$ to $(\sX{})_{\sigma}:=\sX{}$. Now let $\sX{} \in \sieves\fld$ and let $f:\sX{} \to \nat$ be a function. The constant functor also extends to functions by defining $g:=(f)_{\bullet} : \sX{} \to \nat$ by $g(\maxim)(x,\bar m) := f(\maxim)(x)$ for all $\bar m \in \nat^n$ and  all $\maxim \in \fatpoints\fld$. More succintly, we have
\begin{equation}
\sS{}{+}{\sX\ } \subset \sS{}{+}{(\sX\ )_{\bullet}}
\end{equation}
for any $\sX\ \in\sieves\fld$.
\end{example}

Given one of these so-called functions, we can associate a functor $\Gamma_f(-)(\bar m) : \fatpoints\fld \to \set$ defined by 

\begin{equation}
\begin{split}
\maxim\mapsto\Gamma_f(\maxim)(\bar m) &:= \Gamma_{f(\maxim)}(\bar m) \\
&=\{(x,\bar m, n) \mid x\in \sX{[\bar m]}(\maxim), \ f(\maxim)(x,\bar m)= n, \ \mbox{for some} \ n\in \nat\} \ 
\end{split}
\end{equation}
for each $\bar m \in \nat^n$.
Thus, we have an assignment $\Delta^{\circ}\to \nsieves{\fld}{n}$ given by 
\begin{equation}
[n]\mapsto \Gamma_f(\maxim)(\bar m)([n]):=\Gamma_f(\maxim)(\bar m)\cap (\sX{\bar m}(\maxim)\times\{n\}) \ .
\end{equation}
If this assignment is actually a functor (i.e., it is a element of $\nsieves{\fld}{n+1}$), then we say that $ f : \sX{} \to \nat$ is a {\it permissible} $n$-{\it simplicial function}. Given a simpilicial sieve $\sX{}$, we again denote the set of permissible simplicial functions on $\sX\ $ by $\sS{}{+}{\sX{}}$. As before, the operations $\times, \ \sqcup, \ \cup, \ $ and $\cap$, induce operations on $\sS{}{+}{\sX{}}$  via the identification with graphs as before. Furthermore, we have binary operations $+$ and $\cdot$ via pointwise addition and pointwise multiplication which is defined analogously to Definition \ref{addsch} and \ref{multsch}.
For example, given $f, \ g \in \sS{}{+}{\sX{}}$, we define $f+ g$ to be the function which sends a point $(x,\bar m)$ of $\sX{}(\maxim)$ to $f(\maxim)(x,\bar m)+ g(\maxim)(x, \bar m)$ for all $\maxim\in\fatpoints\fld . $

\begin{proposition}
Let $\sX{} \in \nsieves{\fld}{n}$. Then, the assignment which sends an admissible open set $\mathcal{U}$ to $\sS{}{+}{\mathcal{U}}$ is a sheaf of semirings. Thus, we may form the grothendieck ring $\sS{}{}{\mathcal{U}}$ of  $\sS{}{+}{\mathcal{U}}$ with respect to $+$ for each open set  $\sS{}{+}{\mathcal{U}}$. Therefore, the assignment which sends an open set $\mathcal{U}$ to $\sS{}{}{\mathcal{U}}$ is a sheaf of rings, and this gives rise to a morphism of sheaves of semirings $\mbox{gr}(-) : \sS{}{+}{-} \to \sS{}{}{-}$ which induces the grothendieck morphism of semirings $\mathcal{S}_{(x,\bar m)}^{+} \to \mathcal{S}_{(x, \bar m)}$ for any point of $(x,\bar m)$ of $\sX{}(\maxim)$ and for any $\maxim \in \fatpoints\fld$. 
\end{proposition} 

\begin{proof}
The proofs of these statements are essentially the same as the proofs of the corresponding statements in Proposition \ref{persemi} and \ref{persheaf}.
\end{proof}

\begin{example}
Likewise, given an element $\sX\ \in \nsieves{\fld}{n}$, we can define the semiring of positive constant functions $\sK{}{+}{\sX\ }$ (resp., the ring of constant functions $\sK{}{}{\sX\ }$) in a completely analogous way to the schemic definition. Further, the assignment which sends an admissible open set $\mathcal{U}$ to $\sK{}{+}{\mathcal{U}}$ (resp., $\sK{}{}{\mathcal{U}}$) is a sheaf of semirings (resp., a sheaf of rings) on $\sX\ $. Furthermore, $\sS{}{+}{-}$ (resp.,  $\sS{}{}{-}$) is a sheaf of  $\sK{}{+}{-}$-semialgebras (resp., a sheaf of  $\sK{}{}{-}$-algebras) on $\sX\ $.
\end{example}

\section{Total functions} \label{4}

Consider the ring 
\begin{equation}
\bA:=\zet[\lef, \lef^{-1}, (\frac{1}{1-\lef^{-i}})_{i>0}] \ .
\end{equation}
Given an element $\sX\ \in \nsieves{\fld}{n}$, we construct {\it the  ring of total functions on} $\sX\ $, denoted by $\sT{}{}{X}$, by considering the smallest ring generated by elements of $\bA$, elements of $\sS{}{}{\sX \ }$, and objects of the form $\lef^{\alpha}$ where $\alpha$ is an element of $\sS{}{}{\sX \ }$.

For each real number $q>1$, we have an evaluation ring homomorphism
$\nu_q : \bA \to \mathbb{R}$ defined by sending $\lef$ to $q$. Given elements $a$ and $b$ in $\bA$, we say that $a\geq b$ if $\nu_q(a)\geq\nu_q(b)$ for all real $q>1 $. This induces a partial order on $\bA$. We define the following semiring
\begin{equation}
\bA_{+}:=\{ a\in \bA \mid \nu_q(a)\geq 0, \ \forall q >1\} \ .
\end{equation}
This valuation map extends to $\sT{}{}{\sX\ }$ in the natural way.
For $\sX\ \in \nsieves{\fld}{n}$, we define {\it the semiring of total functions}, denoted by $\sT{}{+}{\sX\ }$, by
\begin{equation}
\sT{}{+}{\sX\ }:=\{ f\in\sT{}{}{\sX\ } \mid \nu_q(f)\geq 0, \ \forall q >1\}
\end{equation}
Finally, we define a partial order on $\sT{}{}{\sX\ }$ by setting $f \geq g$ if $f - g \in \sT{}{+}{\sX\ }$.
In the above definitions, we followed \S $4.2$ of \cite{CL}.

\begin{example}
Consider any fat point $\maxim$ over $\fld$. Then, $\sS{}{}{\maxim} = \sK{}{}{\maxim}$. Thus, it is immediate that
\begin{equation}
\sT{}{}{\maxim} \cong \bA \ \ \mbox{and} \ \ \sT{}{+}{\maxim} \cong \bA_{+} \ .
\end{equation}
Let $(-)_{\bullet} : \sieves\fld \to \nsieves{\fld}{n}$ be the trivial functor for some $n > 0$. Then, $\sT{}{}{(\maxim)_{\bullet}}$ is equal to the smallest ring generated by $\bA$, functions from $\nat^{n}$ to $\zet$, and objects $\lef^{\alpha}$ where $\alpha$ is a function from $\nat^{n}$ to $\zet$ modulo the $n$-simplicial identities on the associated graphs.
\end{example}

\begin{example} Fix a 
$\sigma \in \prod_{j=1}^{n-m}\Delta^{\circ}$ where $n, m \in \nat$ with $n\geq m$.
Let $\tau_{\sigma} : \nsieves{\fld}{n} \to \nsieves{\fld}{m}$
be the functor which sends an $n$-simplicial sieve $\sX\ $ to the $m$-simplicial sieve  defined by 
\begin{equation}
\sigma' \mapsto \sX{}(\sigma'\times\sigma ), \ \ \ \forall \sigma' \in \prod_{j=1}^{m}\Delta^{\circ} \ .
\end{equation}
This functor also induces a semiring homomorphism $\tau_{\sigma}: \sT{}{+}{\sX\ } \to \sT{}{+}{\tau_{\sigma}(\sX\ )}$
(and thereby a ring homomorphism  $\tau_{\sigma}: \sT{}{}{\sX\ } \to \sT{}{}{\tau_{\sigma}(\sX\ )}$) defined by sending a positive total function $f: \sX\ \to \nat$ to the unique positive total function $g:=\tau_{\sigma}(f) : \tau_{\sigma}(\sX\ ) \to \nat$ defined by sending a point $(x,\sigma')$ to $f(\maxim)(x,\sigma'\times\sigma)$ for all $\maxim\in \fatpoints\fld$.
\end{example}

\begin{definition}
Let $\sX \ \in \nsieves{\fld}{n}$ and consider an element $f \in \sT{}{}{\sX\ }$. Let $\sY\ := \tau_{\sigma}(\sX\ )$ for some 
$\sigma \in \prod_{j=1}^{n-m}\Delta^{\circ}$ where $n, m \in \nat$ with $n\geq m$. We say that $f$ is $\sigma$-{\it summable over} $\sY\ $ if the sequence
$(\nu_q(\tau_i(f)))_{i\in \nat^{n-m}}$
is summable in $\mathbb{R}$ for all real $q > 1.$
\end{definition}

\begin{definition}
Let $\sX \ \in \nsieves{\fld}{n}$ and consider an element $f \in \sT{}{}{\sX\ }$. Let $\sY\ := \tau_{\sigma}(\sX\ )$ for some 
$\sigma \in \prod_{j=1}^{n-m}\Delta^{\circ}$ where $n, m \in \nat$ with $n\geq m$. We say that $f$ is $\sigma$-{\it integrable over} $\sY\ $
if there exists a function $g \in \sT{}{}{\sY\ }$ such that 
\begin{equation}
\nu_q(g) = \sum_{i\in \nat^{n-m}} \nu_q(\tau_i(f)), \ \ \ \forall q>1 \ .
\end{equation}
If such a $g$ exists, we will denote it by $\mu_{\tau_{\sigma}(\sX\ )}(f)$ or by  $\mu_{\sY \ }^{\sigma}(f)$.
We will denote the subring of $ \sT{}{}{\sX\ }$ formed by all  $\sigma$-integrable functions over $\sY\ $ by $I_{\sY \ }^{\sigma}\sT{}{}{\sX\ }$
\end{definition}

\begin{remark}
We can also define the notion for positive total functions.
Let $\sX \ \in \nsieves{\fld}{n}$ and consider an element $f \in \sT{}{+}{\sX\ }$. Let $\sY\ := \tau_{\sigma}(\sX\ )$ for some 
$\sigma \in \prod_{j=1}^{n-m}\Delta^{\circ}$ where $n, m \in \nat$ with $n\geq m$.
 We say that $f$ is $\sigma$-{\it integrable over} $\sY\ $
if there exists a function $g \in \sT{}{+}{\sY\ }$ such that 
\begin{equation}
\nu_q(g) = \sum_{i\in \nat^{n-m}} \nu_q(\tau_i(f)), \ \ \ \forall q>1 \ .
\end{equation}
If such a $g$ exists, we will denote it by $\mu_{\tau_{\sigma}(\sX\ )}(f)$ or by  $\mu_{\sY \ }^{\sigma}(f)$.
We will denote the subring of $ \sT{}{+}{\sX\ }$ formed by all  $\sigma$-integrable functions over $\sY\ $ by $I_{\sY \ }^{\sigma}\sT{}{+}{\sX\ }$
\end{remark}

\begin{theorem}\label{basicthm}
Let $\sX \ \in \nsieves{\fld}{n}$ and consider an element $f \in \sT{}{}{\sX\ }$.  Let $\sY\ := \tau_{\sigma}(\sX\ )$ for some 
$\sigma \in \prod_{j=1}^{n-m}\Delta^{\circ}$ where $n, m \in \nat$ with $n\geq m$. If
$f$ is $\sigma$-integral over $\sY\ $, then it is $\sigma$-summable over $\sY\ $. Moreover,
the set map
\begin{equation}
\mu_{\sY\ }^{\sigma} : I_{\sY \ }^{\sigma}\sT{}{}{\sX\ } \to \sT{}{}{\sY\ }, \ \ \ f\mapsto \mu_{\sY \ }^{\sigma}(f) 
\end{equation}
is a ring homomorphism.
\end{theorem}

\begin{proof}
This is immediate.
\end{proof}

\begin{remark}
To obtain the converse in \ref{basicthm}, we either have to either investigate a subset of the ring of total functions which are summable or we have to search for an appropriate over-ring of $\bA$.
\end{remark}

\section{Convergence of sequences of total functions} \label{5}
 We say that a sequence of elements $(a_i)$ in $\bA$ is $q$-{\it convergent to} $a_{\star}$ if $(\nu(a_i))$ converges to $(\nu_q(a_{\star}))$ in $\mathbb{R}$ for a fixed a real number $q > 1$. Moreover, we say that  $(a_i)$ is {\it convergent to} $a_{\star}$ if $(\nu(a_i))$ converges to $(\nu_q(a_{\star}))$ in $\mathbb{R}$ for all real $q> 1$.

\begin{lemma} A sequence $(a_i)$ in $\bA$ is convergent to $a_{\star}\in\bA$ if and only if it is $q$-convergent to  $a_{\star}\in\bA$  for some transcendental $q$.
\end{lemma}

\begin{proof}
This follows from the fact that the evaluation ring homomorphism $\nu_q :\bA \to \mathbb{R}$ is injective whenever $q$ is transcendental.
\end{proof}

\begin{example}
Let $(a_i)$ be a sequence such that $a_i=\lef^{n_i}$ for almost all $i$. Assume also that it is $q$-convergent to $a_{\star}\in \bA$ for some real $q> 1$. Then,
\begin{equation}
|\nu_q(a_i) - \nu_q(a_j)| =q^{\min\{n_i,n_j\}}|q^{|n_i -n_j|} - 1|, \ \ \ \forall i,j \ .
\end{equation}
Since $|\nu_q(a_i) - \nu_q(a_j)|$ converges to zero, either $n_i$ diverges to $-\infty$ or $(n_i)$ is a cauchy sequence. However,
a cauchy sequence of integers is eventually constant. Thus, either $n_i$ diverges to $-\infty$ or $n_i = N$ for all sufficiently large $i$ where $N$ is some fixed interger $N$. 
Thus, $a_{\star}$ is either equal to $0$ or $\lef^N$ for some integer $N.$
\end{example}

Let $\sX\ \in \nsieves{\fld}{n}$ and let $(f_i)$ be a sequence of elements in $\sT{}{}{\sX \ }$. We say that $(f_i)$ is  $q$-{\it convergent to} $f_{\star}\in\sT{}{}{\sX \ }$ if for all $\maxim\in\fatpoints\fld$ and all points $(x,\bar m)$ of $\sX{}(\maxim)$, the sequence $(f_i(\maxim)(x,\bar m))$ is $q$-convergent to $f_{\star}(\maxim)(x,\bar m)$ for some fixed real number $q> 1$. Furthermore, we say that  $(f_i)$ is {\it convergent to} $f_{\star}\in\sT{}{}{\sX \ }$ if it is $q$-convergent to $f_{\star}$
for all $q>1$.

\begin{lemma}
Let $\sX\ \in \nsieves{\fld}{n}$ and let $(f_i)$ be a sequence of elements in $\sT{}{}{\sX \ }$. Then, $(f_i)$ is convergent to $f_{\star}\in \sT{}{}{\sX \ }$ if and only if $(f_i)$ is $q$-convergent to $f_{\star}\in \sT{}{}{\sX \ }$ for some transcedental $q$.
\end{lemma}

\begin{example}
Let $(f_i)$ be a sequence of total functions such that $f_i = \lef^{\beta_i}$ where $\beta_i \in \sS{}{}{\sX \ }$ for all $i$. Assume that $(f_i)$ is $q$-convergent to $f_{\star}$ for some fixed $q>1$.
For each $\maxim\in\fatpoints\fld$ and for each point $(x,\bar m)$ of $\sX\ (\maxim)$, we have that $f_{\star}(\maxim)(x,\bar m)$ is either equal to $0$ or $\lef^N$ for some integer $N \in \zet$.
We define a $n$-simplicial subsieve $\mathcal{C}$ of $\sX \ $ by 
\begin{equation}
\mathcal{C}(\maxim) = \mbox{Supp}(f_{\star}(\maxim))
\end{equation}
Let $\mathbbm{1}_{\mathcal{C}} : \sX \ \to \nat $ be the characteristic function of $\mathcal{C}$. Then, $f_{\star}$ is equal to the function 
$\mathbbm{1}_{\mathcal{C}}\cdot\lef^{\beta}$ for some $\beta \in \sS{}{}{\sX \ }$. Again, it is easy to check that $(f_i)$ is convergent to $f_{\star}$ in this case.
\end{example}

Let $\bar \bA$ be the ring $\zet[\lef][[\lef^{-1}]]\supset\bA$ and $\bar \bA_+$ to be the ring largest subring contained in both $\bar\bA$ and $\bA_+$.  For  $\sX\ \in\nsieves{\fld}{n}$, we define
\begin{equation}
\begin{split}
 \bsT{}{}{\sX\ }&= \bar\bA \otimes_{\bA} \sT{}{}{\sX\ } \\
\bsT{}{+}{\sX\ }&= \bar\bA_{+} \otimes_{\bA_{+}} \sT{}{+}{\sX\ }  \ .
\end{split}
\end{equation}

\begin{theorem}
 Let $\sX\ \in\nsieves{\fld}{n}$ and $f \in \sT{}{}{\sX\ }$. Assume that
$f$ is $\sigma$-summable over $\sY \ $ for some $\sigma \in \prod_{j=1}^{n-m}\Delta^{\circ}$ with $n\geq m$. Then, there is a $g\in \bsT{}{}{\sY\ }$ such that 
\begin{equation}
\nu_q(g) = \sum_{i\in \nat^{n-m}} \nu_q(\tau_i(f)), \ \ \ \forall q>1 \ .
\end{equation}
Moreover, $g$ is in $\sT{}{}{\sY\ }$  if and only if $f$ is subject to a presburger condition in the last $n-m$ coordinates.
\end{theorem} 

\begin{proof}
If $f$ is a constant function into $\bar\bA$ or if $f\in\sS{}{}{\sX\ }$, then the statement follows immediately. By properties of limits, this leaves us to prove the statement when $f$ is of the form $\lef^{\beta}$ for some $\beta\in \sS{}{}{\sX\ }$.
Consider the partial sums 
\begin{equation}
s_{i_1,\ldots,i_k}= \tau_{i_1}(f)+\cdots+\tau_{i_k}(f)
\end{equation}
Then for each $\maxim \in\fatpoints\fld$ and each point $(x,\bar m)$ of $\sX\maxim$, we have
\begin{equation}
s_{i_1,\ldots,i_k}(\maxim)(x,\bar m) = \lef^{\beta(\maxim)(x,\bar m \times i_1)}+\cdots+\lef^{\beta(\maxim)(x,\bar m \times i_k)}
\end{equation}
By definition of summability, this implies that $\nu_q(s_{i_1,\ldots,i_k}(\maxim)(x,\bar m))$ converges to some real number $r$.
Therefore, $\nu_q(\beta(\maxim)(x,\bar m \times i_j))$ diverges to $-\infty$ as the sum of the coordinates of $i_j \in \nat^{n-m}$ tends toward $\infty$. This is enough to show that there is an element $a_{\star}(\maxim)(x,\bar m)\in \bar\bA$ such that $\nu_q(a_{\star}(\maxim)(x,\bar m))$ is equal to $r=\lim\nu_q(s_{i_1,\ldots,i_k}(\maxim)(x,\bar m))$. Let $g: \sX\ \to \bar\bA$ be the function defined by $g(\maxim)(x,\bar m) := a_{\star}(\maxim)(x,\bar m)$. It is clear that $g\in \bsT{}{}{\sY\ }$.
The proof of the last assertion is completely analogous to the proof of Theorem-Definition  4.5.1 of \cite{CL}.
\end{proof}

\begin{definition} 
Let $\sX \ \in \nsieves{\fld}{n}$ and consider an element $f \in \sT{}{}{\sX\ }$.  Let $\sY\ := \tau_{\sigma}(\sX\ )$ for some 
$\sigma \in \prod_{j=1}^{n-m}\Delta^{\circ}$ where $n, m \in \nat$ with $n\geq m$. 
We say that a function $f \in  \bsT{}{}{\sX\ }$ is {\it weakly} $\sigma$-{\it integrable over} $\sY\ $ if there is a $g \in \bsT{}{}{\sY\ }$ such that 
\begin{equation}
\nu_q(g)= \sum_{i\in \nat^{n-m}}\nu_q(\tau_i(f))
\end{equation}
for all real $q>1$.
Furthermore if $f \in \bsT{}{+}{\sX\ }$, then we say that it is  \textit{weakly} $\sigma$-\textit{integrable over} $\sY\ $ if there is a $g \in \bsT{}{+}{\sY\ }$ such that 
\begin{equation}
\nu_q(g)= \sum_{i\in \nat^{n-m}}\nu_q(\tau_i(f))
\end{equation}
for all real $q>1$.
In either case, we will often write $\mu_{\sY\ }^{\sigma}(f)$ or $\mu_{\tau_{\sigma}(\sX\ )}(f)$ for $g$. Finally, we will denote the sub-ring (resp., sub-semiring) of $\bsT{}{}{\sX{}}$ (resp. $\bsT{}{+}{\sX{}}$) formed by all weakly $\sigma$-integrable functions over $\sY \ $  by
$I_{\sY\ }^{\sigma}\bsT{}{}{\sX\ }$ (resp., $I_{\sY\ }^{\sigma}\bsT{}{+}{\sX\ }$).
\end{definition}

\begin{remark}
Note that as in the case of Theorem \ref{basicthm}, $\mu_{\sY\ }^{\sigma}$ defines a ring homomorphism (resp., a semiring homomorphism) $\mu_{\sY\ }^{\sigma}:I_{\sY\ }^{\sigma}\bsT{}{}{\sX\ } \to \bsT{}{}{\sY\ }$ (resp.,  $\mu_{\sY\ }^{\sigma}:I_{\sY\ }^{\sigma}\bsT{}{+}{\sX\ } \to \bsT{}{+}{\sY\ }$).
\end{remark}

\begin{theorem}
Let $\sX \ \in \nsieves{\fld}{n}$ and consider an element $f \in \sT{}{}{\sX\ }$.  Let $\sY\ := \tau_{\sigma}(\sX\ )$ for some 
$\sigma \in \prod_{j=1}^{n-m}\Delta^{\circ}$ where $n, m \in \nat$ with $n\geq m$. Then, we have the following injective morphisms
\begin{equation}
\begin{split}
 I_{\sY \ }^{\sigma}\sT{}{+}{\sX \ }\otimes_{\bA_+}\bar\bA_+ &\into  I_{\sY \ }^{\sigma}\bsT{}{+}{\sX \ } \\
 I_{\sY \ }^{\sigma}\sT{}{}{\sX \ }\otimes_{\bA}\bar\bA &\into I_{\sY \ }^{\sigma}\bsT{}{}{\sX \ } \ .
\end{split}
\end{equation}
\end{theorem}

\begin{proof}
Naturally, we investigate the multiplication homomorphism from $I_{\sY \ }^{\sigma}\sT{}{+}{\sX \ }\otimes_{\bA_+}\bar\bA_+$ to $ I_{\sY \ }^{\sigma}\bsT{}{+}{\sX \ }$.
For injectivity, consider the finite sum $\sum f_i \otimes a_i$ with $f_i \in I_{\sY \ }^{\sigma}\sT{}{+}{\sX \ }$ and $a_i \in \bar\bA$. Then, $\sum f_i\cdot a_i = 0$ if and only if $f_i\cdot a_i = 0$ for all $i$. It is immediate then that either $a_i =0$ or $f_i = 0$. Thus, the multiplication semiring homomorphism is injective. The rest follows.
\end{proof}

\section{Functions on limit sieves}\label{6}

\begin{definition} \label{deflimss}
Let $\sX \ \in \nsieves{\fld}{n}$.
Let $\fat \in \limfat \fld$ and choose a point system $\bX$ such that $\fat = \injlim \bX$. For each 
$\maxim\in\bX,$ let $\sX {\maxim} \in \ssieves \fld$ be such that there is natural inclusion
\begin{equation}
\sX{\maxim} \into \sarc_{\maxim}{\sX \ }.
\end{equation}
We define $\sX{\star} := \projlim \sX{\maxim}$ and call  it {\it a limit} $n$-{\it simplicial sieve at the point} $\fat$
{\it with respect to the point system} $\bX$. We call $\sX \ $ {\it the base} of $\sX{\star}$.
\end{definition}

\begin{definition} Let $\sF{-}$ be any of the sheaves of semirings $\sS{}{+}{-}, \sT{}{+}{-}, \bsT{}{+}{-}$ or any of the sheaves of rings $\sS{}{}{-}, \sT{}{}{-}, \bsT{}{}{-}$. Then, for a limit $n$-simplicial sieve $\sX\star$ we define 
\begin{equation}
\sF{\sX\star}:= \injlim \sF{\sX\maxim} \ .
\end{equation}
\end{definition}

\begin{remark}
Let $\sX\star$ be a limit $n$-simplicial sieve.
 Let $\sF{-}$ be any of the sheaves of semirings $\sS{}{+}{-}, \sT{}{+}{-}, \bsT{}{+}{-}$ or any of the sheaves of rings $\sS{}{}{-}, \sT{}{}{-}, \bsT{}{}{-}$. 
Then the assignment which sends an admissible open $\mathcal{U}$ of $\sX\star$ to $\sF{\mathcal{U}}$ is a sheaf of semirings (or rings) on $\sX\star$. Furthermore, if $\mathcal{F}^{+}(-)$ the sheaf of semirings $\sS{}{+}{-}$ (resp., $\sT{}{+}{-}, \bsT{}{+}{-}$) and let $\sF{-}$ the sheaf of rings $\sS{}{}{-}$ (resp., $\sT{}{}{-}, \bsT{}{}{-}$). Then, the direct limit of the grothendieck semiring homomorphisms on each admissible open induces the grothendieck morphism of sheaves of semirings $\sS{}{+}{-} \to \sS{}{}{-}$ (resp., $\sT{}{+}{-} \to \sT{}{}{-}, \bsT{}{+}{-} \to \bsT{}{}{-}$).
\end{remark}

\begin{remark}
All the material from \S 4 and \S 5 carry over easily to the more general case of limit $n$-simplicial sieves.
\end{remark}

Let $\sX\ \in\nsieves{\fld}{n}$. The $n$-simplicial arc operator $\sarc_{\maxim}$ operates on a permissible function $f : \sX \ \to \nat$ (i.e., $f \in \sS{}{+}{\sX\ }$) by acting on the associated graph.  In other words, we may define a set map
\begin{equation}
\sarc_{\maxim}: \sS{}{+}{\sX\ }  \to  \sS{}{+}{\sarc_{\maxim}\sX\ }
\end{equation}
which sends a function $f$ to the unique function $\sarc_{\maxim}f$ whose graph is the $n$-simplicial sieve $\sarc_{\maxim}\Gamma_f$.

\begin{proposition}
Let $\sX\ \in\nsieves{\fld}{n}$. Then, the set map $\sarc_{\maxim}: \sS{}{+}{\sX\ }  \to  \sS{}{+}{\sarc_{\maxim}\sX\ }$ is a semiring homomorphism.
\end{proposition}

\begin{proof}
Let $f, g \in \sS{}{+}{\sX\ }$. First, it is clear that $\sarc_{\maxim}(f \cdot g) = \sarc_{\maxim} f \cdot \sarc_{\maxim}g$.
Consider the decomposition 
\begin{equation}
\Gamma_{f+g}(\fn)([n]) = \bigcup_{t=0}^{n}(\Gamma_{f}(\fn)([t]) \cup \Gamma_{g}(\fn)([n-t])) \ .
\end{equation}
This will almost always be a disjoint union, and in that case, we can just apply Theorem 4.8 of \cite{schmot2} to obtain the result.
In order for this decomposition to not be a disjoint union, there must exists a positive integer $t$ such that $2t = n$ and 
$f(\fn)(x,\bar m) = g(\fn)(x,\bar m)$ for some point $(x,\bar m)$ of $\sX\fn$ for some $\fn \in \fatpoints\fld$.
Thus, suppose this is the case and define a subsieve $\mathcal{C}$ of $\sX \ $ by 
\begin{equation}
\mathcal{C}_{[\bar m]}(\fn) = \{ x \in \sX{[\bar m]}(\fn) \mid f(\fn)(x,\bar m) = g(\fn)(x,\bar m)\} \ .
\end{equation}
Let $h = \mathbbm{1}_{\mathcal{C}}(f- g)$ which is equal to zero by definition where $ \mathbbm{1}_{\mathcal{C}}$ is
the characteristic function of  $\mathcal{C}$. 
Then, $f+ g = f + g + h = (f+ \mathbbm{1}_{\mathcal{C}}f) +  (g-\mathbbm{1}_{\mathcal{C}}g)$,
which means that we may reduce to the case of proving the result for $f+ \mathbbm{1}_{\mathcal{C}}f$. This again reduces further to proving that $\sarc_{\maxim} 2\cdot f = 2 \cdot \sarc_{\maxim}f$. Since the arc operator preserves multiplication and 
$\sarc_{\maxim} k = k$ for any constant function $k$ on $\sX\ $, the result follows. 
\end{proof}

We may define a set map
\begin{equation} \label{arcf}
\sarc_{\maxim}: \sS{}{}{\sX\ }  \to  \sS{}{}{\sarc_{\maxim}\sX\ }
\end{equation}
which sends a function $f$ to the unique function $\sarc_{\maxim}f$ whose graph is the $n$-simplicial sieve $\sarc_{\maxim}\Gamma_f$ modulo the equivalence relation discussed at the end of \S 2.

\begin{proposition}
Let $\sX\ \in\nsieves{\fld}{n}$.
$\sarc_{\maxim}: \sS{}{}{\sX\ }  \to  \sS{}{}{\sarc_{\maxim}\sX\ }$
is the ring homomorphism induced by $\sarc_{\maxim}: \sS{}{+}{\sX\ }  \to  \sS{}{+}{\sarc_{\maxim}\sX\ }$ via the grothendieck semiring homomorphism.
\end{proposition}

\begin{proof}
First, we need to show that Definition \ref{arcf} is well-defined. Thus, let $F$ and $G$ be elements of $\sS{}{}{\sX \ }$ represented by pairs $(f_1,f_2)$ and $(g_1,g_2)$, respectively, which are equivalent. Then, we have defined $\sarc_{\maxim}F$ and $\sarc_{\maxim}G$ to be the function represted by pairs 
$(f_1',f_2')$ and $(g_1',g_2')$, respectively, which have the property that $f_i'$ is the unique function whose graph is $\sarc_{\maxim}\Gamma_{f_i}$ and $g_i'$ is the unique function whose graph is $\sarc_{\maxim}\Gamma_{g_i}$ for $i=1,2$.
From this it follows that $(f_1',f_2')$ and $(g_1',g_2')$ are also equivalent. From this the proposition also follows.
\end{proof}

\begin{remark}
Since $\sarc_{\maxim}$ acts on all of the generators of $\sT{}{}{\sX\ }$ (where we set $\sarc_{\maxim}a = a$ for $a\in \bar\bA$), we also have ring homomorphisms $\sT{}{}{\sX\ } \to \sT{}{}{\sarc_{\maxim}\sX\ }$ and $\bsT{}{}{\sX\ } \to \bsT{}{}{\sarc_{\maxim}\sX\ }$. Likewise, we have semiring homomorphisms $\sT{}{+}{\sX\ } \to \sT{}{+}{\sarc_{\maxim}\sX\ }$ and $\bsT{}{+}{\sX\ } \to \bsT{}{+}{\sarc_{\maxim}\sX\ }$.
\end{remark}

\begin{remark}
 Let $\sF{-}$ be any of the sheaves of semirings $\sS{}{+}{-}, \sT{}{+}{-}, \bsT{}{+}{-}$ or any of the sheaves of rings $\sS{}{}{-}, \sT{}{}{-}, \bsT{}{}{-}$.
Given fat points $\maxim, \fn$ such that $\fn\geq \maxim$, then the relative simplicial arc operator $\sarc_{\maxim/\fn}$ induces a (semi)ring homomorphism 
\begin{equation}
\sarc_{\maxim/\fn} : \sF{\sarc_{\fn}\sX\ }\to \sF{\sarc_{\maxim}\sX\ }
\end{equation}
via sending a function $f$ to the unique function whose graph is of the form $\sarc_{\maxim/\fn}\Gamma_f$. Moreover, this is a directed system.
\end{remark}

\begin{proposition}
 Let $\sF{-}$ be any of the sheaves of semirings $\sS{}{+}{-}, \sT{}{+}{-}, \bsT{}{+}{-}$ or any of the sheaves of rings $\sS{}{}{-}, \sT{}{}{-}, \bsT{}{}{-}$.
Let $\bX$ be a point system. Let $\rho_{\maxim/\fn} : \sarc_{\maxim}\sX\ \to \sarc_{\fn}\sX\ $ be the induced map given by $\fn\leq \maxim$. We write $\rho_{\maxim}$ for $\rho_{\maxim/\spec\fld}$.
The directed system $\rho_{\maxim/\fn}^* : \sF{\sarc_{\fn}\sX\ }\to \sF{\sarc_{\maxim}\sX\ }$
is such that 
\begin{equation}
\sarc_{\maxim}= \rho_{\maxim}^* \ .
\end{equation}
Thus, for any limit $n$-simplicial sieve $\sX\star$ and any $f\in\sF{\sX\star}$, we have
\begin{equation}
f = \sarc_{\fat}g
\end{equation}
where $g\in \sF{\sX\maxim}$ for some $\maxim\in \bX$.
\end{proposition}

\begin{proof}
immediate.
\end{proof}

\begin{proposition}
Let $\sX\ $ be an $n$-simplicial sieve and let $\maxim\in\fatpoints\fld$ and let $\sY\ := \tau_{\sigma}(\sX\ )$ for some 
$\sigma \in \prod_{j=1}^{n-m}\Delta^{\circ}$ where $n, m \in \nat$ with $n\geq m$.  If $f$ is (weakly) $\sigma$- integrable over $\sY\ $, then $\sarc_{\maxim}f$ is (weakly) $\sigma$-integrable over $\sY\ $. Moreover,
\begin{equation}
\mu_{\sY\ }^{\sigma}(\sarc_{\maxim}f) = \sarc_{\maxim}\mu_{\sY\ }^{\sigma}(f) \ .
\end{equation}
\end{proposition}
\begin{proof}
immediate.
\end{proof}

Let $\sX\ $ be a limit $n$-simplicial sieve,  $\bX$ its points system with $\fat=\injlim\bX$, and let $\sY\ := \tau_{\sigma}(\sX\ )$ for some 
$\sigma \in \prod_{j=1}^{n-m}\Delta^{\circ}$ where $n, m \in \nat$ with $n\geq m$. 
If $f$ is in $I_{\sY\ }^{\sigma}\bsT{}{}{\sX\ }$ (or, $I_{\sY\ }^{\sigma}\sT{}{}{\sX\ }$), then 
\begin{equation}
\mu_{\sY\ }^{\sigma}(f) = \nabla_{\fat}\mu_{\sY\ }^{\sigma}(g)
\end{equation}
for some $g \in \bsT{}{}{\sX\maxim }$ (resp., $\sT{}{}{\sX\maxim }$) where $\maxim\in\bX$. Therefore, we have the following proposition.

\begin{proposition}
Let $\sX\ $ be a limit $n$-simplicial sieve. Let $\sY\ := \tau_{\sigma}(\sX\ )$ for some 
$\sigma \in \prod_{j=1}^{n-m}\Delta^{\circ}$ where $n, m \in \nat$ with $n\geq m$. Then, there are isomorphisms
\begin{equation}
\begin{split}
&I_{\sY\ }^{\sigma}\sT{}{+}{\sX\ } \cong \injlim I_{\sY\ }^{\sigma}\sT{}{+}{\sX\maxim } \\
&I_{\sY\ }^{\sigma}\sT{}{}{\sX\ } \cong \injlim I_{\sY\ }^{\sigma}\sT{}{}{\sX\maxim } \\
&I_{\sY\ }^{\sigma}\bsT{}{+}{\sX\ } \cong \injlim I_{\sY\ }^{\sigma}\bsT{}{+}{\sX\maxim } \\
&I_{\sY\ }^{\sigma}\bsT{}{}{\sX\ } \cong \injlim I_{\sY\ }^{\sigma}\bsT{}{}{\sX\maxim } \ .
\end{split}
\end{equation}
\end{proposition}

\section{Cach\' e functions}

Let $\mot$ be a partial $n$-simplicial  motivic site (which we will simply call a motivic site for brevity) over a limit $n$-simplicial sieve $\mathcal{Z}$.
In \cite{me2}, I defined the notion of the grothendieck rings $\grot{\mot}$ and $\grot{\mmot}$
where $\mmot$ is the motivic site of all measurable limit $n$-simplicial sieves in $\mot$.

An $n$-simplicial subsieve $\sY\ $ of $\mathcal{Z}$ that is an element of $\mot$ may not be measurable. We form the ring $\sT{0}{}{\sZ\ }$  by taking as generators the functions of the form $\mathbbm{1}_{\sY\ }$ where $\sY\ $ is a measurable  $n$-simplicial subsieve of $\mathcal{Z}$ and the element $\lef -1$. We also define 
$\bsT{0}{}{\sX\ }:=\sT{0}{}{\sX\ }\otimes_{\bA}\bar\bA$ . Note that 
there is an injective ring homomorphism  $\sT{0}{}{\sZ\ } \into \grot{\mmot}$ defined by sending $\mathbbm{1}_{\sY\ }$ to
$[\sY\ ]$. More specifically, the ring structure on $\sT{0}{}{\sZ\ }$ is the restriction of the ring structure on $\grot{\mmot}$.

We define the ring of {\it cach\'e functions on} $\mathcal{Z}$ {\it over} $\bA$  {\it with respect to} $\mot$ by 
\begin{equation}
\sC{}{}{\mathcal{Z},\mot}:= \grot{\mmot}\otimes_{\sT{0}{}{\sZ{}}} \sT{}{}{\sZ\ }
\end{equation}
and the ring of {\it cach\'e functions  on} $\mathcal{Z}$ {\it over} $\bar\bA$  {\it with respect to} $\mot$  by 
\begin{equation}
\bsC{}{}{\mathcal{Z},\mot} :=\sC{}{}{\mathcal{Z},\mot}\otimes_{\bA}\bar\bA\ .
\end{equation}
For ease of notation, we will sometimes write $\sC{}{}{\sZ\ }$ for $\sC{}{}{\mathcal{Z},\mot}$ and $\bsC{}{}{\sZ\ }$ for $\bsC{}{}{\mathcal{Z},\mot}$ as the underlying motivic site will usually be explicitly stated.

We also define the ring of $\sigma$-{\it integrable cach\'e functions on} $\mathcal{Z}$ {\it over} $\sY\ =\tau_{\sigma}(\sX \ )$  {\it with respect to} $\mot$ and coefficients in $\bA$ by 
\begin{equation}
I_{\sY\ }^{\sigma}\sC{}{}{\mathcal{Z},\mot}:= \grot{\tau_{\sigma}(\mmot)}\otimes_{\sT{0}{}{\sY{}}} I_{\sY\ }^{\sigma}\sT{}{}{\sZ\ }
\end{equation}
and the ring of $\sigma$-{\it integrable cach\'e functions on} $\mathcal{Z}$ {\it over} $\sY\ =\tau_{\sigma}(\sZ \ )$  {\it with respect to} $\mot$ and coefficients in $\bar\bA$ by
\begin{equation}
I_{\sY\ }^{\sigma}\bsC{}{}{\mathcal{Z},\mot}:= \grot{\tau_{\sigma}(\mmot)}\otimes_{\bsT{0}{}{\sY{}}} I_{\sY\ }^{\sigma}\bsT{}{}{\sZ\ } \ .
\end{equation}
Again, we will sometimes write $I_{\sY\ }^{\sigma}\sC{}{}{\sZ\ }$ for $I_{\sY\ }^{\sigma}\sC{}{}{\mathcal{Z},\mot}$ and $I_{\sY\ }^{\sigma}\bsC{}{}{\sZ\ }$ for $I_{\sY\ }^{\sigma}\bsC{}{}{\mathcal{Z},\mot}$ when the underlying motivic site is fixed.
Note that we have an inclusion of rings $ I_{\sY\ }^{\sigma}\sC{}{}{\mathcal{Z}}\into I_{\sY\ }^{\sigma}\bsC{}{}{\mathcal{Z}}$.

\begin{lemma}
Let $\mot$ be a motivic site relative to $\sZ\ \in \nsieves{\fld}{n}$. Then, there is a surjective morphism of motivic sites
\begin{equation}
\tau_{\sigma}(\mmot)\to\categ{Mes}\tau_{\sigma}(\mot)
\end{equation}
for any $m$-simplicial complex $\sigma$ with $n\geq m$. Thus, we have a surjective ring homomorphism
\begin{equation}
\grot{\tau_{\sigma}(\mmot)}\to\grot{\categ{Mes}\tau_{\sigma}(\mot)} \ .
\end{equation}
\end{lemma}

\begin{proof}
immediate.
\end{proof}

In general, the assignment which sends an admissible open $\mathcal{U}$ of $\sZ\ $ to $\sF{\mathcal{U}}$ is not a sheaf on $\sZ\ $. Thus, we will denote the sheafification of $\sF{-}$ by $^a\sF{-}$.

\begin{theorem}
Let $\sZ\ $ be a limit simplicial sieve and let $\sF{-}$ be the any of the functors of the form
$\sC{}{}{-,F(-)},  \bsC{}{}{-,F(-)}, I_{-}^{\sigma}\sC{}{}{-,F(-)}, I_{-}^{\sigma}\bsC{}{}{-,F(-)}$ where $F$ is any functor  from limit $n$-simplicial sieves to the category of motivic sites which are closed under infinite union. Then
$\grot{\mmot_{F(\sZ{})}|_{-}}$ is a flasque sheaf on $\sZ\ $. Also, the assignment which sends an admissible open $\mathcal{U}$ of $\sZ\ $ to $\sF{\mathcal{U}}$ is a presheaf on $\sZ\ $. 
\end{theorem}

\begin{proof}
We just need to prove that $S(-):=\grot{\mmot_{F(\sZ{})}|_{-}}$ is in fact a sheaf as the rest of the statements are trivial.
Therefore, let $\mathcal{U}$ be an admissible open of $\sZ\ $ and let $\mathcal{U}_i$ be an arbitrary cover of $\mathcal{U}$ by admissible opens. For each, $i$, let $s_i \in S(\mathcal{U}_i)$ be such that $s_i|_{\mathcal{U}_j} = s_j|_{\mathcal{U}_i}$.
If $s_i = [\mathcal{S}_i]$ for some $n$-simplicial sieve $\mathcal{S}_i \in F(\mathcal{U}_i)\subset F(\mathcal{U})$, then this translates to 
\begin{equation}
[\mathcal{S}_i\cap\mathcal{U}_j]=[\mathcal{S}_j\cap\mathcal{U}_i]
\end{equation}
in $S(\mathcal{U})$.
We define the gluing of the $\mathcal{S}_i$ to be the element $[\mathcal{S}]$ where $\mathcal{S}=\cup \mathcal{S}_i$. Such a sieve is clearly measurable. In this vain, proving unicity amounts to showing that if
\begin{equation}
 [\mathcal{S}\cap\mathcal{U}_i]=[\mathcal{T}\cap\mathcal{U}_i] \ ,
\end{equation}
then $[\mathcal{S}]=[\mathcal{T}]$ which is easily shown to be the case by using the scissor relation.
\end{proof}

\subsection{Analogues of the results in \S 6}
From now on let  $\sF{-}$ be the any of the presheaves of the form
$\sC{}{}{-,F(-)}$,  $\bsC{}{}{-,F(-)}$, $I_{-}^{\sigma}\sC{}{}{-,F(-)},$ or $I_{-}^{\sigma}\bsC{}{}{-,F(-)}$ where $F$ is a functor which sends an $n$-simplicial sieve to the category of motivic sites which are closed under arbitrary union and such that 
\begin{equation}
\sarc_{\maxim}\mathcal{S} \in \categ{Mes}F(\sarc_{\maxim}\sZ{} ), \ \ \ \mbox{for all } \mathcal{S} \in \categ{Mes}F(\sZ\ ) \ . 
\end{equation}
We have
\begin{equation}
\sF{\sZ\ } \cong \injlim \sF{\sZ\maxim} \ ,
\end{equation}
and  the simplicial arc operator $\sarc_{\maxim}$ defines a morphism of ringed sieves from $(\sZ\maxim , \sF{-})$ to
$(\sZ0, \sF{-})$ where $\sZ0$ is the base of $\sZ\ $ by tensoring  the action of $\sarc_{\maxim}$ on the grothendieck ring of the motivic site $\categ{Mes}F(\sZ0)$ and the action of  $\sarc_{\maxim}$ on the other factor as defined in \S 6.

 Thus, every element $f$ of $\sF{\sZ\ }$ is of the form 
\begin{equation}
f = \sarc_{\fat}g
\end{equation}
for some $g\in \sF{\sZ\maxim}$ and for some $\maxim \in \bX$ where $\fat =\injlim \bX$ is the point system defining $\sZ{}$. Thus, if $f$ is an element of $I_{\sY\ }^{\sigma}\sC{}{}{\sZ \ ,F(\sZ{})}$ or $I_{\sY\ }^{\sigma}\bsC{}{}{\sZ \ ,F(\sZ{})}$, we have
\begin{equation}
\mu_{\sY\ }^{\sigma}(f) = \sarc_{\fat}\mu_{\sY\ }^{\sigma}(g)
\end{equation}
for some $g$ in  $\sC{}{}{\sZ\maxim ,F(\sZ\maxim)}$ (or, respectively,  $\bsC{}{}{\sZ\maxim ,F(\sZ\maxim)}$) for some $\maxim \in \bX$.

\begin{remark}
An analogous theorem to Theorem 10.1.1 of \cite{CL} holds for both $I_{-}^{\sigma}\sC{}{}{-,F(-)}$ and $I_{-}^{\sigma}\bsC{}{}{-,F(-)}.$ We will take this subject up in \cite{me3}.
\end{remark}

\subsection{Spealization to measurability}

Here we  apply the notion of the category of families of sieves indexed by $\nat^n$  as per the comments in \S $7.3$ of \cite{me2}. We denote this category by  $\categ{i}^n\categ{Sieves}_{\fld}$,

Note that there are a large amount of measures coming from the grothendieck ring  of a motivic site. 
By choosing a family of limit sieves $\sZ{}$, we are fixing the point system $\bX$ and limit point $\fat$ for each element of a motivic site $\mot$ relative to $\sZ{}$. We fix an ultra-filter $\sim$ on $\bX$ and a non-negative real number $Q$. We denote by $\mu_Q$ the cooresponding measure defined by
\begin{equation}\label{defrelmes}
\mu_{\fat, \bX, Q}^{\sim}(\sX{\star}) :=  \ulim {\class{\sX \maxim}\lef^{-\lceil Q\cdot\dim {\sarc_{\maxim}{\sX \ }}\rceil}} \ .
\end{equation}
Restricting to the sub-motivic site $\mot_Q$ of elements of $\mot$ which are measurable with respect to $\sim$ and $Q$, then
$\mu_{Q}$ restricted to $\grot{\mot_Q}_{\lef_{\bullet}}$ factors through the composition
\begin{equation}
\grot{\mot_Q}_{\lef_{\bullet}}\to\grot{\ssieves {S}}_{\lef_{\bullet}} \to \prod_{\sim} \grot{\ssieves {S}}_{\lef_{\bullet}}
\end{equation}
 where the map on the far right is the diagonal homomorphism. We can realize this as a cach\'e function $I_{\sX{}}$ defined inductively by $\beta_{\maxim}:=\lceil Q\cdot\dim{\sarc_{\maxim}{\sX \ }}\rceil$ on $\sX{}(\maxim)\setminus \sX{}(\maxim')$ and $\beta_{\maxim} =  \beta_{\maxim'}$ on $\sX{}(\maxim')$ where $\maxim' \leq \maxim$ 
and
\begin{equation}
I_{\sX{}}(\maxim)=\mathbbm{1}_{\sX{\maxim}}\cdot\lef^{-\beta_{\maxim}} \ .
\end{equation}
for all $\maxim \in \bX$ and $0$ otherwise.
Note that $\beta_{\maxim} \in \sS{}{}{\sX{}}$. Clearly then, since $\sX{}\in\mot_{Q}$, we have that $I_{\sX{}}$ is $\sigma$-integrable for all $\sigma$. In this way, we can see that the notion of integrability in this paper specializes to the notion of measurability \cite{me2}. 

Now, we define
\begin{equation} \label{pair}
\int _{\sX{}}^{\sigma} d\mu_{Q}:=\int^{\sigma} \mathbbm{1}_{\sX{}} d\mu_{Q} := \mu_{\sX{\sigma}}(I_{\sX{}}) \ .
\end{equation}
For each $\sigma$, we can extend this linearly through elements of the pullback of the the ring of $\sigma$-integrable  total functions on $\sZ{}$ along the structure morphism. We may also extend linearly through elements of the grothendieck ring.
Thus, we have defined a pairing of $\nat^n$ and $\cap_{\sigma} I_{\sZ{\sigma}}\sC{}{}{\sZ{},\mot_Q}$ which sends $(\sigma, c)$ to the element 
\begin{equation}
\int^{\sigma}c \in \sT{}{}{\sX{\sigma}}
\end{equation}
 determined by Equation \ref{pair}. We call this element the geometric  integral of $c$ along the $\sigma$-simplex.
This is a natural extension of the idea of measure in \cite{me2}. It is important to note that we may apply the pushforward of the structure morphism $j$ to $\sT{}{}{\sX{\sigma}}$ and tensoring by $1$ gives us an element of $\sC{}{}{\sZ{\sigma}}$. It could be interesting to investigate the properties of this integral relative to the action of the measure $\mu_{\sZ{\sigma}}$ more closely. The important thing to note is that if $\mu_{\sZ{\sigma}}$ factors in this way, then we are actually computing something like a schemic geometric motivic integral in the vein of \cite{schmot2}, \cite{me1}, or \cite{me2}.

\subsection{Topological realization}

In \cite{me2}, we defined the notion of the topological motivic site. Likewise, geometric realization should turn a total function into functions from a CW-complex to $\bA$. This is possibly something to investigate further.

We also defined the notions of homotopy grothendieck rings and topological homotopy grothendieck rings. Thus, we may tensor the ring of total functions on $\sX{}$ by the homotopy grothendieck ring to obtain the notion of the ring of homotopy cach\'e functions $\sC{}{h}{\sX{}}$ which will come equipped with a surjective ring homomorphism
\begin{equation}
\sC{}{h}{\sX{}}\to \sC{}{}{\sX{}} \ .
\end{equation}

\end{document}